\documentclass[11pt,reqno]{amsart}  
\usepackage{amsmath, amsthm, amscd,amssymb,datetime} 
\usepackage[all, arrow]{xy}\SelectTips{cm}{10}
\usepackage{hyperref,todonotes,mathtools} %%% <--- It's useful for working on drafts. We can
\mathtoolsset{showonlyrefs, showmanualtags}
\def\@setcopyright{}
\def\serieslogo@{}
\makeatother

\numberwithin{equation}{section}

\newcommand{\tors}[1]{{#1}_{\rm tors}}

\newcommand{\IG}{{\mathbb G}}

\newcommand{\IQ}{\mathbb{Q}}
\newcommand{\IQbar}{\overline{\IQ}}
\newcommand{\IZ}{{\mathbb Z}} 
\newcommand{\IN}{{\mathbb N}}

\newcommand{\IF}{{\mathbb F}}
\newcommand{\IFbar}{\overline{\IF}}
\newcommand{\sa}[1]{{#1}^{sa}}
\newcommand{\ab}[1]{{#1}^{ab}}

\newcommand{\cyc}[1]{{#1}^{cyc}}

\newcommand{\gl}[2]{{\rm GL}_{#2}({#1})}
\renewcommand{\sl}[2]{{\rm SL}_{#2}({#1})}
\newcommand{\psl}[2]{{\rm PSL}_{#2}({#1})}
\newcommand{\mat}[2]{{\rm Mat}_{#2}({#1})}

\newcommand{\gal}[1]{{\rm Gal}({#1})}
\newcommand{\aut}[1]{{\rm Aut}({#1})}
\newcommand{\cO}{\mathcal{O}}
\newcommand{\cG}{\mathcal{G}}
\newcommand{\ex}[1]{{\rm ex}({#1})}

\DeclareMathOperator{\disc}{disc}

\begin{document}
\baselineskip=14.5pt

\bibliographystyle{plain}

\makeatletter
\def\imod#1{\allowbreak\mkern10mu({\operator@font mod}\,\,#1)}
\makeatother
\newcommand{\vectornorm}[1]{\left|\left|#1\right|\right|}
\newcommand{\ip}[2]{\left\langle#1,#2\right\rangle}
\newcommand{\ideal}[1]{\left\langle#1\right\rangle}
\newcommand{\sep}[0]{^{\textup{sep}}}
\newcommand{\Span}[0]{\operatorname{Span}}
\newcommand{\Tor}[0]{\operatorname{Tor}}
\newcommand{\Stab}[0]{\operatorname{Stab}}
\newcommand{\Orb}[0]{\operatorname{Orb}}
\newcommand{\Soc}[0]{\operatorname{Soc}}
\newcommand{\Kernel}[0]{\operatorname{ker }}
\newcommand{\Gal}[0]{\operatorname{Gal}}
\newcommand{\Aut}[0]{\operatorname{Aut}}
\newcommand{\Out}[0]{\operatorname{Out}}
\newcommand{\Inn}[0]{\operatorname{Inn}}
\newcommand{\lcm}[0]{\operatorname{lcm}}
\newcommand{\Image}[0]{\operatorname{im }}
\newcommand{\vol}[0]{\operatorname{vol }}

\newcommand{\pgl}[0]{\operatorname{PGL}_2(\mathbb{F}_3)}
\newcommand{\GL}[0]{\operatorname{GL}}
\newcommand{\GLnz}[2]{\GL_#1(\ZZ/#2\ZZ)}
\newcommand{\SL}[0]{\operatorname{SL}}
\newcommand{\AGL}[0]{\operatorname{AGL}}
\newcommand{\PGL}[0]{\operatorname{PGL}}
\newcommand{\PSL}[0]{\operatorname{PSL}}
\newcommand{\PSp}[0]{\operatorname{PSp}}
\newcommand{\PSU}[0]{\operatorname{PSU}}
\newcommand{\Sp}[0]{\operatorname{Sp}}
\newcommand{\POmega}[0]{\operatorname{P\Omega}}
\newcommand{\Uup}[0]{\operatorname{U}}
\newcommand{\Gup}[0]{\operatorname{G}}
\newcommand{\mcFup}[0]{\operatorname{F}}
\newcommand{\Eup}[0]{\operatorname{E}}
\newcommand{\Bup}[0]{\operatorname{B}}
\newcommand{\Dup}[0]{\operatorname{D}}
\newcommand{\Mup}[0]{\operatorname{M}}
\newcommand{\mini}[0]{\operatorname{min}}

\newcommand{\maxi}[0]{\textup{max}}
\newcommand{\modu}[0]{\textup{mod}}
\newcommand{\nth}[0]{^\textup{th}}
\newcommand{\sd}[0]{\leq_{sd}}
 \newtheorem{theorem}{Theorem}[section]
 \newtheorem{proposition}[theorem]{Proposition}
 \newtheorem{lemma}[theorem]{Lemma}
 \newtheorem{corollary}[theorem]{Corollary}
 \newtheorem{conjecture}[theorem]{Conjecture}
 \newtheorem{definition}[theorem]{Definition}
 \newtheorem{question}[theorem]{Question}
 \newtheorem*{claim*}{Claim}
 \newtheorem{claim}[theorem]{Claim}
 \newtheorem{example}[theorem]{Example}
 \newtheorem*{example*}{Example}
 \newtheorem{remark}{Remark}
 \newtheorem*{remark*}{Remark}
 \newcommand{\mc}{\mathcal}
 \newcommand{\mf}{\mathfrak}
 \newcommand{\mcA}{\mc{A}}
 \newcommand{\mcB}{\mc{B}}
 \newcommand{\mcC}{\mc{C}}
 \newcommand{\mcD}{\mc{D}}
 \newcommand{\mcE}{\mc{E}}
 \newcommand{\mcF}{\mc{F}}
 \newcommand{\mcG}{\mc{G}}
 \newcommand{\mcH}{\mc{H}}
 \newcommand{\mcI}{\mc{I}}
 \newcommand{\mcJ}{\mc{J}}
 \newcommand{\mcK}{\mc{K}}
 \newcommand{\mcL}{\mc{L}}
 \newcommand{\mcM}{\mc{M}}
 \newcommand{\mcN}{\mc{N}}
 \newcommand{\mcO}{\mc{O}}
 \newcommand{\mcP}{\mc{P}}
 \newcommand{\mcQ}{\mc{Q}}
 \newcommand{\mcR}{\mc{R}}
 \newcommand{\mcS}{\mc{S}} 
 \newcommand{\mcT}{\mc{T}}
 \newcommand{\mcU}{\mc{U}}
 \newcommand{\mcV}{\mc{V}}
 \newcommand{\mcW}{\mc{W}}
 \newcommand{\mcX}{\mc{X}}
 \newcommand{\mcY}{\mc{Y}}
 \newcommand{\mcZ}{\mc{Z}}
 \newcommand{\mfp}{\mf{p}}
 \newcommand{\mfP}{\mf{P}}
 \newcommand{\mfq}{\mf{q}}
 \newcommand{\mfQ}{\mf{Q}} 
 \newcommand{\mfd}{\mf{d}}
 \newcommand{\mfa}{\mf{a}}
 \newcommand{\mfb}{\mf{b}}
 \newcommand{\oQ}{\overline{\QQ}}
 \newcommand{\AAA}{\mathbb{A}}
 \newcommand{\CC}{\mathbb{C}}
 \newcommand{\FF}{\mathbb{F}}
  \newcommand{\GG}{\mathbb{G}}
 \newcommand{\NN}{\mathbb{N}}
 \newcommand{\PP}{\mathbb{P}}
 \newcommand{\QQ}{\mathbb{Q}}
 \newcommand{\QQd}{\QQ^{(d)}}
 \newcommand{\RR}{\mathbb{R}}
 \newcommand{\ZZ}{\mathbb{Z}}
 \newcommand{\hhat}{\hat{h}}
 \newcommand{\oneto}[1]{\{1,\dots,#1\}}
 \newcommand{\QQbar}[0]{\overline{\QQ}}
 \newcommand{\toisom}[0]{\overset{\sim}{\longrightarrow}}
   \def\bbbone{{\mathchoice {\rm 1\mskip-4mu l} {\rm 1\mskip-4mu l}
   {\rm 1\mskip-4.5mu l} {\rm 1\mskip-5mu l}}}

\newcommand{\frob}{{\rm Fr}}

 \renewcommand{\thefootnote}{\fnsymbol{footnote}}
 % 1   2    3     4     5     6     7     8     9 
 % * \dag                          **

\title[Small points and free abelian groups]{Small points and free abelian groups}
\author[Robert Grizzard]{Robert Grizzard}
\address{\newline University of Texas at Austin \newline Department of Mathematics, RLM 8.100 \newline 2515 Speedway, Stop C1200 \newline Austin, TX 78712-1202}
\email{rgrizzard@math.utexas.edu}
\author[Philipp Habegger]{Philipp Habegger}
\address{\newline{Fachbereich Mathematik}\newline
{Technische Universit\"at Darmstadt}\newline
 Schlossgartenstra\ss e 7 \newline
{64289 Darmstadt, Germany}}
\email{habegger@mathematik.tu-darmstadt.de}
\author[Lukas Pottmeyer]{Lukas Pottmeyer}
\address{\newline{Fachbereich Mathematik}\newline
{Technische Universit\"at Darmstadt}\newline
 Schlossgartenstra\ss e 7 \newline
{64289 Darmstadt, Germany}}
\email{pottmeyer@mathematik.tu-darmstadt.de}
\today~~\currenttime~~CDT

\begin{abstract}
Let $F$ be an algebraic extension of the rational numbers and $E$ an
elliptic curve defined over some number field contained in $F$. The
absolute logarithmic Weil height, respectively the N\'eron-Tate
height, induces a norm on $F^*$ modulo torsion, respectively on $E(F)$
modulo torsion. The groups $F^*$ and $E(F)$ are free abelian modulo
torsion if the height function does not attain arbitrarily small positive values. In this paper we prove the failure of the converse to this statement by explicitly constructing counterexamples.
\end{abstract}

\maketitle

\tableofcontents

%%%%%%%%%%%%%%%%%%%%%%%%%%%%%%%%%%%%%%%%%%%%%%%%%%%%%%%%%%%%%%%%%%%%
\section{Introduction}

Throughout the text we fix an algebraic closure
$\QQbar$ of $\QQ$, and all algebraic extensions of $\QQ$ are assumed
to be subfields of $\QQbar$. One can ask for which fields $F$ the
multiplicative group $F^*$ is free modulo torsion, 
 we call an abelian group $G$ free modulo torsion if 
$G/\tors{G}$
is a free abelian group where
 $\tors{G}$ denotes the torsion subgroup of $G$.
In the rational case
$\QQ^*$ is free modulo torsion as $\IZ$ is a unique factorization
domain. More generally, using classical ideal factorization theory and
Dirichlet's Unit Theorem  one can prove that $K^*$ is free modulo torsion for any number
field $K$. 

Before we recall more advanced examples, we have to fix some notation. For any subfield $F\subseteq\IQbar$, let $F^{ab}$ denote the maximal abelian extension of $F$, and $F^{(d)}$ denote the compositum of all extensions of $F$ of degree at most $d$. Moreover, we let $\QQ^{tr}$  denote the maximal totally real field extension of $\QQ$.

Iwasawa \cite{Iwasawa1953} proved  that $(K^{ab})^*$ is free modulo
 torsion.
 Some years later Schenkman \cite{Schenkman1964} showed that $(\QQ^{(d)})^*$ is free modulo
 torsion for all positive integers $d$. 
 May \cite{May72} rediscovered Iwasawa's result and combined it with
 Schenkman's result to show that $((K^{(d)})^{ab})^*$ is free modulo
 torsion for all number fields $K$ and all positive integers $d$. Another class of fields $F$ such
 that $F^*$ is free modulo torsion consists of all Galois extensions
 of $\QQ$ which contain only finitely many roots of unity. Horie's paper
 \cite{Horie1990} contains this result, it 
appears to
 be the origin. We  immediately see that $(\QQ^{tr})^*$ is free modulo torsion.

A  related problem is to find algebraic extensions $F/K$
such that the Mordell-Weil group $E(F)$ is free modulo torsion for a
given elliptic curve $E$ defined over  $K$. 
This is clearly the case by the Mordell-Weil Theorem if $F$ is a number field.
Here too the interest lies in infinite extensions $F/K$. 

We now introduce some notation to unify the multiplicative and elliptic cases. 
Let $\mcG$ denote either the algebraic torus $\mathbb{G}_m$ or an
elliptic curve defined over a number field $K$. We will usually
suppose $F\supseteq K$.

If $\mcG = \GG_m$ is a torus, we take the absolute logarithmic Weil
 height to be the canonical height on $\mcG$.
 If $\mcG$ is an elliptic curve,  the canonical height on $\mcG$ is
 understood to be the N\'{e}ron-Tate height $\widehat{h}:\mcG(\QQbar) \to
 [0,\infty)$.  For the definitions and basic properties of these
 heights, we refer to Bombieri and Gubler's book \cite{BG}.  
 Our N\'eron-Tate height is twice the height used by Silverman
in \S 9, Chapter VIII \cite{Silverman:AEC}. 
The canonical height is well-defined modulo
 torsion, i.e. it factors through to a mapping
 $\mcG(\QQbar)/\tors{\mcG(\QQbar)}\rightarrow [0,\infty)$. 
We observe that the group
 $\mcG(\QQbar)/\tors{\mcG(\QQbar)}$ is divisible and torsion
 free. Thus it carries the structure of a
 $\QQ$-vector space.
If $\mcG=\IG_m$, then the canonical height is a norm on this vector
 space. 
In the case of an elliptic curve, its square root $\widehat{h}^{1/2}$
 is one. 

The basis of our investigation is the apparent coincidence that, apart from $(K^{(d)})^{ab}$, all fields we  
discussed above are known to share another property, the Bogomolov
property,  related to the Weil
height. We briefly describe this property. 

We say that $F$ has the Bogomolov property with respect to
$\mcG$ if there exists $\epsilon > 0$ such that the
 canonical height of a non-torsion points of $\mcG(F)$ is
at least $\epsilon$. 
We recall that torsion points are exactly the
points of height zero.
This property was coined by Bombieri and Zannier \cite{BZ2} who worked
in the multiplicative setting.

The fields $K^{ab}$, $\QQ^{(d)}$ and $\QQ^{tr}$ have the
Bogomolov property relative to $\mcG$ 
\cite{AmorosoDvornicich}, \cite{az1}, \cite{bakersilverman}, \cite{BZ2}, \cite{Baker2005}, \cite{Schinzel1973}, \cite{ZhangEquidist}.
It is an open question, posed
in more general form by Amoroso, David, and Zannier   \cite[Problem 1.4]{ADZ:propB},
 whether $(K^{(d)})^{ab}$ has
the Bogomolov property with respect to $\mcG$.
Recently, the second-named author established another class of fields
having the Bogomolov property \cite{Habegger2013}. If $\mcG$ is an elliptic curve defined
over $\QQ$ then the field generated over $\IQ$ by $\tors{\mcG(\IQbar)}$ has the Bogomolov property with
respect to the torus and $\mcG$.

Note that neither the Bogomolov property of a field $F$ nor the property that
$\mcG(F)/\tors{\mcG(F)}$ is free abelian is preserved under finite
extensions $F'/F$. A counterexample for both properties
when $\mcG=\IG_m$  is the extension $\QQ^{tr}
(i)/\QQ^{tr}$. 
That the stated properties are not preserved in this extension was
first observed  by Amoroso and Nuccio \cite{Amoroso2007} and
May \cite[Example 1]{May72}, respectively. A counterexample in the elliptic curve case is presented by the third-named author, cf. \cite[Example 5.7]{PottmeyerRamifi}.

Recall that a norm on an abelian group is called discrete if zero is
 an isolated value of
its image.
Lawrence \cite{lawrence} and Zorzitto \cite{zorzitto} showed that a
 countable abelian group is free abelian if and only if it admits a
 discrete norm. 
The countability condition was later removed by
 Stepr{\=a}ns \cite{steprans}, but  the groups considered in this paper are
 countable. These results  immediately imply the
 following proposition.

\begin{proposition}\label{bfree}
If $F$ is a subfield of $\QQbar$ with $F\supseteq K$ that satisfies the Bogomolov property with
respect to $\mcG$, then $\mcG(F)/\tors{\mcG(F)}$ is free abelian.
\end{proposition}

Our aim is to  discuss the failure of the converse of this
statement. We will prove that the converse does not hold by explicitly
constructing counterexamples in the cases where $\mcG$ is
$\IG_m$,  an elliptic curve with complex multiplication (CM), and for an arbitrary elliptic curve defined over $\QQ$. In other words, in these cases we construct fields $F$ where $\mcG(F)/\tors{\mcG(F)}$ is free abelian, but there are points of arbitrarily small positive canonical height on $\mcG(F)$.
Here and in the rest of this paper, an elliptic curve is said to have
CM over $K$ if the ring of endomorphisms of $E$ which are defined over
$K$ is strictly larger than $\IZ$. The elliptic curve has CM if it has
CM over some number field.

Now we can formulate our main result.

\begin{theorem}\label{thm:intro}
Let $\mcG$ be $\IG_m$ or an elliptic curve. We suppose that $\mcG$ is
 defined over a number field $K$. 
There is an algebraic extension $F/K$ such that
 $\mcG(F)/\tors{\mcG(F)}$ is free abelian but $F$ does not have the
 Bogomolov property with respect to $\mcG$ 
\begin{enumerate}
\item 
if $\mcG$ is $\IG_m$,
\item
if $\mcG$ is an elliptic curve with CM over $K$, and
\item
if $\mcG$ is an elliptic curve without CM, and $K = \QQ$.  In this case $\tors{\mcG(F)}$ is finite.
\end{enumerate}
\end{theorem}

We will give a more precise description of $F$ below. We refer to
Proposition \ref{prop:Gmexplicit} for part (1) and
Proposition \ref{prop:ellcurveexplicit}  for parts (2) and (3).

The proof in Theorem \ref{thm:intro} that the group in question is free abelian  is given
in Section \ref{free} after providing two criteria for
the freeness of abelian groups. 
Applying the criteria involves investigating  the structure of various
Galois groups.
For example,   part (3) requires information on the Galois group
of $K(\tors{\mcG(\IQbar)})/K$. Therefore, 
Serre's \cite{Serre72} famous theorem plays an important role in our
argument. 

In the remaining sections we show that the field $F$ does not satisfy
the Bogomolov property with respect to the canonical height on $\mcG$
in  the three cases described in Theorem \ref{thm:intro}. To do
this, we explicitly describe how to construct points of arbitrarily
small height which are defined over the field $F$.
In the multiplicative setting we work with roots of 
\begin{equation}
\label{eq:osadapoly}
X^n-X-1\quad\text{if $n\ge 2$.}
\end{equation}
These roots have small Weil height as $n$ tends to infinity by basic
height inequalities. Moreover,
Osada \cite{Osada} proved that the Galois group attached to the
splitting field of \eqref{eq:osadapoly} over the rationals is the
full symmetric group $S_n$. This is enough Galois theoretic
information to apply one of the two criteria mentioned above. Indeed,
we will
conclude that our roots are  in a common field whose multiplicative
group is free modulo torsion. 

The basic approach for an elliptic curve $E$  is similar but more involved. 
The roots now come from a polynomial equality, not unsimilar to
\eqref{eq:osadapoly}, that involves the multiplication-by-$n$ endomorphism
of $E$. The difficulty here lies in determining the Galois group of the
 associated splitting field. To facilitate this we introduce a new
 variable $T$ and assume that $n$ is
 the power of a prime $p$.
Moreover, we will suppose that $E$ has
 supersingular reduction above a place with residue characteristic
 $p$. The
polynomial reduced modulo a place above $p$ lies in
$\IFbar_p(T)[X]$, with $\IFbar_p$ an algebraic closure of $\IF_p$, and takes on a particularly simple form, it
is a trinomial.
We then use a result of Abhyankar \cite{Abhyankar} to conclude that the reduced
polynomial  is irreducible. This
result also provides
sufficient information on its Galois group.

Back in characteristic zero we will use  variants of Hilbert's
irreducibility theorem due to Dvornicich-Zannier \cite{DvZa07} and
Zannier \cite{Za10}. This
allows us
to specialize $T$ to a root of unity. 
This will often preserve irreducibility and the Galois structure of
the splitting field without increasing the height too much. The latter 
observation is due to the fact that  torsion points have canonical
height zero.
 The criteria mentioned above apply again. 

In the non-CM case the curve $E$ has supersingular reduction at
infinitely many primes by a theorem of Elkies \cite{Elkies87}. 
In the CM case we have infinitely many admissible primes by more
classical results. 
 In both cases we obtain a sequence of  points with small N\'eron-Tate
 height which are sufficient for Theorem \ref{thm:intro}.

As a byproduct of our labor we
exhibit  infinite extensions of the rationals over which a given
elliptic curve without CM has only finitely many points of finite order,
cf. part 1 of Lemma \ref{lem:finitetornonCM}.

The authors thank Will Sawin for answering a question on mathoverflow
leading to the proof of Lemma \ref{will}, and Paul Fili for providing
the reference \cite{Iwasawa1953}.  Thanks are also due to Ga\"{e}l
R\'{e}mond and Umberto Zannier for providing feedback on an earlier draft of this paper,
and to Jeffrey Vaaler, who  observed
Proposition \ref{bfree} as a consequence of the
Lawrence-Zorzitto-Stepr{\=a}ns Theorem.
The second-named author was partially supported by the National Science Foundation under
agreement No.~DMS-1128155. Any opinions, findings and conclusions or
recommendations expressed in this material are those of the authors
and do not necessarily reflect the views of the National Science
Foundation. 
The third-named author was supported by the DFG project ``Heights and
unlikely intersections'' HA~6828/1-1.
The  authors thank the Institute for Advanced Study in Princeton  for providing a stimulating and encouraging environment.

%%%%%%%%%%%%%%%%%%%%%%%%%%%%%%%%%%%%%%%%%%%%%%%%%%%%%%%%%%%%%%%%%%%%
\section{Some Group Theory}\label{free}
\subsection{Free Abelian Criteria}

Recall that a subgroup $H$ of an abelian group $G$ is called
\emph{pure} if $G/H$ is torsion-free.  The following version of
Pontryagin's result on free abelian groups is proved (in a stronger form) in \cite[Theorem 2.3, Chapter IV]{eklofmekler}.

\begin{theorem}[Pontryagin's Criterion]\label{Pontryagin}
Let $G$ be a countable abelian group.  The following are equivalent:
\begin{enumerate}
\item $G$ is free abelian;
\item every finite subset of $G$ is contained in a pure free abelian subgroup of $G$;
\item every finite rank subgroup of $G$ is free abelian.
\end{enumerate}
\end{theorem}
From this result the following two useful criteria are easily derived.  They will be applied in proving Theorem \ref{thm:intro}.

\begin{proposition}[Criterion A]\label{fintors}
Let $\mcG$ be $\IG_m$ or an elliptic curve. We suppose that $\mcG$ is
 defined over a number field $K$ and assume that $F$ is a Galois extension of $K$.
  If $\tors{\mcG(F)}$ is finite, then $\mcG(F)/\tors{\mcG(F)}$ is free abelian.  
\end{proposition}

\begin{proof}
  We refer to \cite[Proposition 1]{Horie1990} for the $\GG_m$ case, and to \cite[Proposition 3]{Moon12} for the case of an elliptic curve (or, more generally, an abelian variety).
 \end{proof}

\begin{proposition}[Criterion B]\label{finreltors}
Let $\mcG$ be $\IG_m$ or an elliptic curve. We suppose that $\mcG$ is
 defined over a number field $K$,
 and let $F_0$ and $F$ be algebraic extensions of $K$ with $K\subseteq F_0 \subseteq F$.  If every finite subextension of $F/F_0$ is contained in a finite subextension $M$ of $F/F_0$ such that 
\begin{enumerate}
\item[(a)] $\mcG(M)/\tors{\mcG(M)}$ is free abelian, and 
\item[(b)] the torsion subgroup of $\mcG(F)/\mcG(M)$ has finite exponent,
\end{enumerate}
then $\mcG(F)/\tors{\mcG(F)}$ is free abelian.
If, furthermore, $F/F_0$ is Galois, then $\mcG(F)/\tors{\mcG(F)}$ is free abelian if (1) and (2) are satisfied for all finite Galois extensions $M/F_0$, with $M\subseteq F$.
\end{proposition}

\begin{proof}
The proof of the $\GG_m$ case is originally from May (see \cite[Lemma 1]{May72}). However, his proof applies also if $\mcG$ is an elliptic
curve (or even an abelian variety). For convenience we will give the proof in detail here.

We want to use the implication $(2)\Longrightarrow (1)$ in
Pontryagin's Criterion \ref{Pontryagin}. Let $S=\{\overline{P_1},\dots,\overline{P_r}\}$ be a finite subset of
$\mcG(F)/\tors{\mcG(F)}$ with each $P_i\in \mcG(F)$. The field $F_0(P_1,\dots,P_r)$ is a
finite extension of $F_0$, and thus contained in a finite extension
$M/F_0$ with $M\subseteq F$ satisfying (a) and (b). This construction yields $S\subseteq \mcG(M)/\tors{\mcG(M)}$. Let $m$ be the exponent of $\tors{\left( \mcG(F)/\mcG(M) \right)}$. Then 
\begin{equation}
H=\left\{\overline{P} \in \mcG(F)/\tors{\mcG(F)}:\,\, [m]\overline{P} \in \mcG(M)/\tors{\mcG(M)}\right\}
\end{equation}
is a pure subgroup of $\mcG(F)/\tors{\mcG(F)}$. Moreover $H$ is free abelian, since $\mcG(M)/\tors{\mcG(M)}$ is free abelian, by assumption (a). 
Now Pontryagin's Criterion yields the stated result as $S$ is
contained in $H$.

Let $F/F_0$ be Galois and $M/F_0$ a finite extension, with $M\subseteq F$. Then the Galois closure of $M$ over $F_0$ is still contained in $F$. Thus the second statement follows immediately from the first one.
\end{proof}

\subsection{Some Facts on Matrix Groups}
\label{sec:matgrpfacts}
We collect  some facts on 
matrix groups $G=\gl{\IZ/p^n\IZ}{2}$
where 
 $p$ will denote a prime number and $n\ge 1$ an integer.  For any finite group $H$, we denote its exponent by $\ex{H}$.  In what follows we repeatedly use the classical Jordan-H\"older Decomposition Theorem without mentioning it directly. 

 The group $G$
 lies in the short exact sequence
\begin{equation}
  1\rightarrow U^{(1)} \rightarrow G 
  \rightarrow
  \gl{\IF_p}{2}\rightarrow 1
\end{equation}
where
\begin{equation}
  U^{(1)} = 1 + p \mat{\IZ/p^{n}\IZ}{2}.
\end{equation}
So $U^{(1)}$ is a  $p$-group of order $p^{4(n-1)}$ or trivial. Therefore
\begin{equation}
\label{eq:orderG}
  \vert G\vert
 =  p^{4(n-1)} \vert \gl{\IF_p}{2}\vert = 
(p^2-1)(p^2-p)p^{4(n-1)}. 
\end{equation}
We  generalize $U^{(1)}$ by setting
\begin{equation}
  U^{(k)} = \left\{1 + p^k B: B\in \mat{\IZ/p^n\IZ}{2}\right\}
\end{equation}
for $1\le k\le n$. 
Each $U^{(k)}$ is a normal subgroup of $G$ of order
$p^{4(n-k)}$ and lies in the short exact sequence
\begin{equation}
\label{eq:Ukexact}
  1\rightarrow U^{(k)} \rightarrow G 
  \rightarrow
  \gl{\IZ/p^k\IZ}{2}\rightarrow 1.
\end{equation}

Over the prime field we find the exact sequences
\begin{equation}\label{eq:gl2sequence}
\begin{aligned}
  1&\rightarrow \sl{\IF_p}{2}\rightarrow\gl{\IF_p}{2}\rightarrow
  \IF_p^*\rightarrow 1,
\quad\text{and}\\
1&\rightarrow\{\pm 1\}\rightarrow
\sl{\IF_p}{2}\rightarrow\psl{\IF_p}{2}\rightarrow 1 .
\end{aligned}
\end{equation}

Recall that a $p$-group only has $\IZ/p\IZ$ as composition factors
and that $\psl{\IF_p}{2}$ is simple if $p\ge 5$.
We recall that $\psl{\IF_p}{2}$ are solvable if $p\le 3$.
So if $p\ge 5$, then $\psl{\IF_p}{2}$ is a  composition factor of
$G$ and all  other composition factors are abelian.

What about the exponent of $G$? The determinant is a  homomorphism
\begin{equation}
  G \rightarrow (\IZ/p^n\IZ)^*
\end{equation}
onto the group of units of $\IZ/p^n\IZ$. 
The unit group  is  cyclic of order $p^{n-1}(p-1)$
if $p\ge 3$. In this case, the exponent $\ex{G}$ is a
multiple of $p^{n-1}(p-1)$. 
If $p=2$ and $n\ge 2$, then $(\IZ/2^n\IZ)^*$ contains a cyclic subgroup of
order $2^{n-2}$, and $\ex{G}$ is a multiple of $p^{n-2}=2^{n-2}$.

\begin{lemma}
\label{lem:boundH}
If $H$ is a subgroup of $G$, then
$\vert H\vert <p^{4+8{\rm ord}_p\ex{H}}\le p^4 \ex{H}^8$.
\end{lemma}
\begin{proof}
  We factor $\vert H\vert =p^e m$, where $p\nmid m$ and $e\ge 0$.  
As $p^e m$ divides $\vert G\vert = p^{4(n-1)}(p^2-1)(p^2-p)$, we see 
that $m$  divides $(p^2-1)(p-1)$, and therefore 
\begin{equation}
m < p^3. 
\end{equation}

To bound $p^e$ we may assume $e\ge 1$. Let $H'$ be a
$p$-Sylow subgroup of $H$, so for some $f \ge 1$ we have $p^f = \ex{H'} \mid \ex{H}$. 

If $p^f = 2$ we set $k=2$ and otherwise we take $k=f$; note that $k\ge f$.
We claim that $U^{(k)}\cap H'$ contains at most $p^{4f}$ elements. 

Indeed, say
$1+p^k B$ lies in this intersection where
$B\in\mat{\IZ/p^n\IZ}{2}$. Then
expanding the right-hand side of $1=(1+p^kB)^{p^f}$ and subtracting $1$ yields
\begin{equation}
\label{eq:expand}
  0=
\sum_{l=1}^{p^f}
{p^f\choose l}p^{kl} B^{l}=
 p^{k+f}B\left(1 + \sum_{l=2}^{p^f}
{p^f\choose l}p^{kl-k-f} B^{l-1}\right);
\end{equation}
we observe $kl-k-f\ge 2k-k-f\ge 0$ in the sum.
If $l>2$, then $kl-k-f>0$, so
\begin{equation}
p\text{ divides } {p^f\choose l}p^{kl-k-f} 
\end{equation}
for $l \in \{3,\ldots,p^f\}$ and even if $l=2$. 
Therefore, the matrix in brackets on the right of \eqref{eq:expand} 
lies in $U^{(1)}\subseteq\gl{\IZ/p^n\IZ}{2}$, and we have
\begin{equation}%\label{pkf}
 p^{k+f}B=0.
\end{equation}
If $k\ge n$, then $1+p^kB$ was trivial to start out with. 
Otherwise,
we may represent the entries of $B$ with integers in $[0,p^{n-k})$. 
If
 $k+f\le n$ then these entries
 are divisible by $p^{n-k-f}$, so there are at most $p^{4f}$
possibilities for $B$. 
If finally $k < n < k+f$,
then the number of possible $B$ is at most $p^{4(n-k)}<p^{4f}$.
Our claim holds true.

According to the exact sequence \eqref{eq:Ukexact}
the quotient $H'/ (U^{(k)}\cap H')$ is isomorphic to a subgroup of
$\gl{\IZ/p^k\IZ}{2}$ and it is a $p$-group or trivial. 
We recall \eqref{eq:orderG} and the bound for $\vert U^{(k)}\cap H'\vert$ from
above to find 
\begin{equation}
 \vert H'\vert \le 
    \vert U^{(k)}\cap H' \vert \cdot p^{4(k-1)+1}
\le p^{4f+4(k-1)+1}
\le  p^{4(k+f)-3}.
\end{equation}
As $k\le f+1$ we conclude
$|H'| \le p^{8f+1} \le p \cdot \ex{H}^8$. 

We have $\vert H'\vert =p^e$.
Taking the product of the bounds for $m$ and $p^e$ yields the lemma.
\end{proof}

\begin{lemma}
\label{lem:noAn}~
\begin{enumerate}
\item Let $H$ be a subgroup of $G$. Then any non-abelian composition
factor of $H$ is isomorphic to $\psl{\IF_p}{2}$ for some prime $p\ge 5$.
\item If $n\ge 6$ is an integer and $p\ge 5$ is a prime, then
$A_n$ and $\psl{\IF_p}{2}$ are not isomorphic. 
\end{enumerate}
\end{lemma}
\begin{proof}
For the first claim let $C$ be a non-abelian composition factor of
$H$. 
The kernel of $\gl{\IZ/p^n\IZ}{2}\rightarrow \gl{\IF_p}{2}$ is
trivial or a $p$-group. 
As $C$ is not abelian it is also
 a composition factor of the image of $H$ in $\gl{\IF_p}{2}$.
The two exact sequences \eqref{eq:gl2sequence}
indicate that $C$ is a composition factor of a subgroup of $\psl{\IF_p}{2}$.
So $p\ge 5$ because $\psl{\IF_2}{2}$ and $\psl{\IF_3}{2}$ are both
solvable. 
Of course, $C$ could be isomorphic to the full group $\psl{\IF_p}{2}$,
which is simple. This  case is covered in the first part of the lemma. 
Otherwise, 
 $C$ is a quotient of a proper subgroup of $\psl{\IF_p}{2}$.
Dickson classified all possible subgroups of $\psl{\IF_p}{2}$, see \cite[Hauptsatz II.8.27]{Huppert67}.
According to this classification a proper subgroup is either solvable or isomorphic to
$A_5$. 
As $C$ is simple and non-abelian it must be isomorphic to  
$A_5\cong\psl{\IF_5}{2}$.
We have established the first claim in the lemma. 

The second claim follows since 
$\psl{\IF_p}{2}$ is not isomorphic to an alternating group if $p\ge 7$; Artin
\cite{Artin55}
gives a simple proof by comparing cardinalities. 
\end{proof}

We conclude with the following observation about the commutator subgroup $[H,H]$ of an arbitrary subgroup $H$ of $G$.
\begin{lemma}
\label{lem:commutator}
  Let $H$ be a subgroup of $G$, and set $d=[G:H]!$. Then 
$\vert [H,H]\vert \ge p^{n/3}/d$.
\end{lemma}
\begin{proof}
Since $d=[G:H]!$, the $d$-th power of any element in $G$ lies in $H$. We define
\begin{equation}
  A=\left(
  \begin{array}{cc}
    1 & dx \\ 0 & 1
  \end{array}\right)\quad\text{and}\quad
  B=\left(
  \begin{array}{cc}
    1 & 0 \\ dx & 1
  \end{array}\right)
\end{equation}
for $x\in\IZ/p^n\IZ$.
Then $A,B\in H$ and 
\begin{equation}
  ABA^{-1}B^{-1}=
\left(
  \begin{array}{cc}
    1 +(dx)^2+(dx)^4& -(dx)^3 \\ (dx)^3 & 1-(dx)^2
  \end{array}\right) \in [H,H].
\end{equation}
If $x$ runs over the elements represented by integers in
$[0,p^{n/3}d^{-1})$
then $ABA^{-1}B^{-1}$ runs over a set of distinct elements in $[H,H]$, and the
lemma follows.
\end{proof}

%%%%%%%%%%%%%%%%%%%%%%%%%%%
\subsection{Applications}%\label{exb}

If $F$ is a subfield of $\IQbar$, then $\sa{F}$ denotes the field obtained by adjoining to $F$ all roots of
irreducible polynomials in $F[X]$ with symmetric or alternating Galois
groups, of any degree.

\begin{proposition}\label{freethings}
Let $\mcG$ be $\IG_m$ or an elliptic curve. We suppose that $\mcG$ is
 defined over a number field $K$.
 Let $d$ be a positive integer, and let $F = \left(\left(K^{ab}\right)^{sa}\right)^{(d)}$.  Then $\mcG(F)/\tors{\mcG(F)}$ is free abelian 
\begin{enumerate}
\item if  $\mcG$ is $\IG_m$,
\item if $\mcG$ is an elliptic curve with CM over $K$, and
\item if
 $\mcG$ is an elliptic curve without CM.  In this case
  $\tors{\mcG(F)}$ is finite.
\end{enumerate}
\end{proposition}

For a more general version of (3) regarding elliptic curves without 
CM we refer to part (1) of Lemma \ref{lem:finitetornonCM} below.

\begin{lemma}\label{will}
Let $\mcG$ be $\IG_m$ or an elliptic curve. We suppose that $\mcG$ is
 defined over a number field $K$.
Let $m$ be a positive integer  and $F/K$  an algebraic extension such
that all $m$-torsion points of $\mcG$ are defined over $F$. If $P$ is
an algebraic point of $\mcG$ whose image in $\mcG(\QQbar)/\mcG(F)$ has
order $m$, then $F(P)/F$ is a finite abelian Galois extension, and $\Gal(F(P)/F)$ contains an element of order $m$.
\end{lemma}

\begin{proof}
Let $[m]$ denote the multiplication-by-$m$ endomorphism of $\mcG$. 
If $[m]P = Q \in \mcG(F)$, then the conjugates of $P$ over $F$ are
other solutions of this equation, and they differ from $P$ by
$m$-torsion points. In particular, since all of these torsion points
are defined over $F$, we know that $F(P)/F$ is Galois.
We have a homomorphism $\Phi: \Gal(F(P)/F) \to \mcG[m]$ given by
$\sigma \mapsto \sigma(P) - P$ with target the points $\mcG[m]$ in
$\mcG(\IQbar)$ of order dividing $m$.  The
image of $\Phi$ is not contained in $\mcG[n]$ for any $n < m$. If it
were, then   $[n]P$ would be fixed by each element of $\Gal(F(P)/F)$,
meaning $[n]P \in \mcG(F)$, which is contrary to our assumption that
$P$ has order $m$ in $\mcG(\QQbar)/\mcG(F)$.
By the Elementary Divisor Theorem the image of $\Phi$, which is a finite abelian group,  contains an
element of order $m$. The lemma follows as $\Phi$ is injective. 
\end{proof}

We call a finite group \emph{admissible} if it has no composition factor which is abelian or isomorphic to $\psl{\IF_p}{2}$ for any $p\ge 5$. We call a field extension admissible if it is Galois and its Galois group is admissible. 
Note that the compositum of two admissible extensions of a common base field is again admissible over the base field. A Galois subextension of an admissible extension is also admissible. 

The exponent of a finite Galois extension is the exponent of its Galois group. The compositum of two finite Galois extensions of exponent dividing some $e\in\IN$ is again a finite Galois extension with exponent dividing $e$.

\begin{lemma}
\label{lem:finitetornonCM}
Let $\mcG$ be $\IG_m$ or an elliptic curve. Assume that $\mcG$ is
 defined over a number field $K$, and let $k,e\in\IN$.
Suppose
\begin{equation}
\ab{K}= F_0 \subseteq F_1 \subseteq F_2 \subseteq \cdots\subseteq
F_{2k-1}\subseteq F_{2k} \subseteq \IQbar
\end{equation}
is a tower of field extensions with the following property:
If $1\le i\le 2k$ then 
 $F_{i}$ is generated 
by all  admissible extensions of $F_{i-1}$ in $\IQbar$ for even $i$, and
$F_{i}$ is generated by all finite Galois 
extensions of $F_{i-1}$ in $\IQbar$ 
with exponent dividing $e$ for odd $i$. Then the following properties hold.
\begin{enumerate}
\item  If $\mcG=E$ is an elliptic curve without CM, then $\tors{E(F_{2k})}$ is finite and 
$E(F_{2k})/\tors{E(F_{2k})}$ is free abelian. 
\item If $M$ is a finite Galois subextension of $F_{2k}/F_0$ and if $\tors{\cG}\subseteq \cG(F_0)$, then the exponent of 
$\tors{\left(\cG(F_{2k})/\cG(M)\right)}$  divides $e^k$ and $\mcG(F_{2k})/\tors{\mcG(F_{2k})}$ is free abelian.
\end{enumerate}
\end{lemma}

\begin{proof}
All extensions $F_{i}/F_{i-1}$ are Galois
 and $F_0/K$ is Galois too. 
We show by induction on $i$ that $F_i/K$ is Galois. 
If $i\ge 1$ and if $\sigma:F_i\rightarrow\IQbar$
is the identity on $K$ then $\sigma(F_{i-1})=F_{i-1}$ and
$\sigma(F_i)/F_{i-1}$ is generated by Galois extensions with certain
group-theoretic properties. Hence $\sigma(F_i)=F_i$ and thus $F_i/K$
is Galois.

The second claim in part (1) follows from the first one and Criterion A, Proposition \ref{fintors}. 
We now prove the first claim in (1). 

Let $P \in E(F_{2k})$ be a point of finite order. % $p^n$. 
It suffices to show that the said order is bounded
independently of $P$. Without loss of generality we may assume that
the order is $p^n$ for some prime $p$ and some integer $n\ge 1$. 

We construct inductively an auxiliary tower of number fields by first
taking $K_{2k}$ to be the normal closure of $K(P)/K$ in $\IQbar$, so
$K_{2k}\subseteq F_{2k}$. 
As $K(E[p^n])/K$ is normal we have $K_{2k}\subseteq K(E[p^n])$. 
For any $1\le i\le 2k$ we construct $K_{i-1}$ using the diagram

\begin{equation}\label{eq:diagram}
\begin{xy}
(-12,0)*+{K_i}="Ki";
(12,0)*+{F_{i-1}}="Fim1";
(0,12)*+{K_i F_{i-1}}="comp";
(0,-12)*+{K_{i-1}}="Kim1";
(15,-12)*+{=K_i \cap F_{i-1}}="e";
 (0,24)*+{F_i}="Fi";
 (0,-24)*+{K}="K";
  {\ar@{-} "comp";"Fi"};
 {\ar@{-} "Kim1";"Ki"};
 {\ar@{-} "Kim1";"Fim1"};
 {\ar@{-} "Ki";"comp"};
 {\ar@{-} "Fim1";"comp"};
 {\ar@{-} "K";"Kim1"};
 \end{xy}
\end{equation}

By induction we find that $K_{i-1}/K$ is Galois.
We will repeatedly
use the fact that $K_i/K_{i-1}$ is Galois and that 
the restriction induces an isomorphism between
 $\gal{K_iF_{i-1}/F_{i-1}}$ and $\gal{K_i/K_{i-1}}$

We now prove
$K_{i}=K_{i-1}$ for even $i$.
According to the short exact sequence
\begin{equation}
  1\rightarrow \gal{K_{i}/K_{i-1}} \rightarrow
  \gal{K_{i}/K}\rightarrow \gal{K_{i-1}/K}\rightarrow 1
\end{equation}
 any composition factor of $\gal{K_{i}/K_{i-1}}$ is a composition
factor
of $\gal{K_{i}/K}$. It is thus a composition factor
of $H=\gal{K(E[p^n])/K}$, because the latter group restricts onto
$\gal{K_{i}/K}$. By \eqref{eq:diagram} and the hypothesis of this lemma
$\gal{K_{i}/K_{i-1}}$ is admissible. 
We can identify $H$ with a subgroup of $\gl{\IZ/p^n\IZ}{2}$, and so we have 
$K_{i}=K_{i-1}$ by Lemma \ref{lem:noAn}, part (1). 

We are left with a contracted tower
$K_0 \subseteq K_2\subseteq K_4\subseteq\cdots\subseteq K_{2k}$
such that ${K_{i}/K_{i-2}}$  has exponent dividing $e$ for even $i$. 
So $\ex{\gal{K_{2k}/K_0}}\mid e^k$. 

We abbreviate 
$\Gamma=\gal{K_{2k}/K}$, which is a quotient of $H$.
 Serre's Theorem \cite{Serre72} implies that there is a constant $p_0$ depending only on $E$
such that $\Gamma = \gl{\IZ/p^n\IZ}{2}$ if $p> p_0$. 
Moreover, if $p\le p_0$ then $[\gl{\IZ/p^n\IZ}{2}:H]$ and $\vert \ker
(H\rightarrow\Gamma) \vert$ are bounded
independently of $n$. 

Let us assume  $p> \max\{4,p_0,e^k\}$. We have a short exact sequence
\begin{equation}
  1\rightarrow \gal{K_{2k}/K_0}\rightarrow \Gamma = \gl{\IZ/p^n\IZ}{2}
  \rightarrow \gal{K_0/K}\rightarrow 1. 
\end{equation}
The group $\gal{K_0/K}$ is abelian since $K_0\subseteq F_0 = \ab{K}$. 
So $\psl{\IF_p}{2}$ is a composition factor 
of $\gal{K_{2k}/K_0}$ by the discussion in Section
\ref{sec:matgrpfacts}.
We conclude 
$\ex{\psl{\IF_p}{2}}\le \ex{\gal{K_{2k}/K_0}}\le e^k$. 
As
\begin{equation}
  \left(
  \begin{array}{cc}
    1 & 1 \\ 0 & 1
  \end{array}
\right) 
\end{equation}
has order $p$ as an element of $\psl{\IF_p}{2}$
we find $p \le e^k$, contradicting our assumption on $p$.  Therefore, $p \ll 1$, where the implied constant here and below only
depends on $E,K,e$ and $k$. 

It remains to show $n\ll 1$. 
Let $H'$ denote the preimage
of $\gal{K_{2k}/K_0}$ under the quotient map
$H\rightarrow \Gamma$.
We use the conclusion of Serre's Theorem 
together with
\begin{equation}
\ex{H'}\le \vert \ker(H\rightarrow\Gamma) \vert \ex{\gal{K_{2k}/K_0}}
\le \vert \ker(H\rightarrow\Gamma) \vert e^k
\end{equation}
to conclude $\ex{H'}\ll 1$. 
As $p$ is bounded too, Lemma \ref{lem:boundH} 
implies $\vert H'\vert \ll 1$. 
Recall that 
$H/H' \cong \Gamma / \gal{K_{2k}/K_0}\cong \gal{K_0/K}$
is abelian. So $ [H,H]\subseteq H'$ and in particular
$\vert [H,H]\vert \le \vert H'\vert \ll 1$. 
We apply Lemma \ref{lem:commutator} to $H$ and recall
 $[\gl{\IZ/p^n\IZ}{2}:H]\ll 1$,
to conclude $n\ll 1$. This concludes the proof of (1).

For part (2) we introduce the auxiliary fields $F'_i = F_i M$ for
$0\le i\le 2k$. Since $M/F_0$ is a finite Galois extension, so is $F'_i/F_i$.
As $F_i/F_{i-1}$ is Galois if $i\ge 1$, we conclude the same for
$F'_i/F'_{i-1}$. 
Moreover, ${\rm Gal}(F'_i/F'_{i-1})$ is isomorphic to
${\rm Gal}(F_i/F'_{i-1}\cap F_i)$, and $F'_{i-1}\cap F_i/F_{i-1}$ is
finite and Galois. 
By hypothesis, a finite Galois extension of
$F'_{i-1}$ inside $F'_i$ is admissible if $2\mid i$ and 
has exponent dividing $e$ if $2\nmid i$. We will apply this
observation in just a moment.

Now suppose $m$ is the order of an element in  $\tors{(\cG(F_{2k})/\cG(M))}$
which is represented by $P\in \cG(F_{2k})$. We abbreviate $L = M(P)$; this
is a subfield of $F_{2k}$ and it fits into the diagram
\begin{equation}\label{eq:diagram2}
\begin{xy}
(-12,0)*+{F'_{i-1}}="Ki";
(12,0)*+{L\cap F'_i}="Fim1";
(0,12)*+{F'_{i-1} \left(L\cap F'_i\right)}="comp";
(0,-12)*+{L \cap F'_{i-1}}="Kim1";
 (0,24)*+{F_i'}="Fi";
 (0,-24)*+{M}="K";
  {\ar@{-} "comp";"Fi"};
 {\ar@{-} "Kim1";"Ki"};
 {\ar@{-} "Kim1";"Fim1"};
 {\ar@{-} "Ki";"comp"};
 {\ar@{-} "Fim1";"comp"};
 {\ar@{-} "K";"Kim1"};
 \end{xy}
\end{equation}

By Lemma \ref{will} the extension $L/M$ is finite abelian, and its exponent is a
multiple of $m$; here we have used our assumption that $\cG(F_0)$ contains all torsion
points of $\cG$. Therefore, $L\cap F'_i/M$ is abelian and so is
$F'_{i-1}(L\cap F'_i)/F'_{i-1}$ by the diagram. 
We recall the observation above. If $2\mid i$, then 
$F'_{i-1}(L\cap F'_i)/F'_{i-1}$ is also admissible. But only the
trivial group is admissible and abelian. Therefore, the extension is
trivial and so $L\cap F'_i = L\cap F'_{i-1}$. 
As in the proof of (1) our tower of fields contracts, i.e. 
\begin{equation}
  M = L\cap F'_0 \subseteq L\cap F'_2 \subseteq
  \cdots \subseteq L\cap F'_{2k} = L\cap F_{2k} = L. 
\end{equation}
If on the other hand we have $2\nmid i$, then 
$L\cap F'_i/L\cap F'_{i-1}$ has exponent dividing $e$. We conclude
that 
$L/M$ has exponent dividing $e^k$ and thus $m\mid e^k$ by the
conclusion of Lemma \ref{will}.

We want to apply Criterion B, Proposition \ref{finreltors}, in order
to prove that $\mcG(F_{2k})/\tors{\mcG(F_{2k})}$ is free abelian. 
We have already verified part (b) of the hypothesis. The extension $F_{2k}/F_0$ is Galois. Therefore, we are left to show that $\mcG(M)/\tors{\mcG(M)}$ is free abelian for every finite Galois extension $M/F_0$, with $M\subseteq F_{2k}$. %% Let $M$ be as described.
We claim that $M$ satisfies the Bogomolov property with respect to $\mcG$, then Proposition \ref{bfree} concludes the proof.

Let $M=K^{ab}(\alpha)$, then $M$ is contained in an abelian extension of the number field $K(\alpha)$. In the case of $\GG_m$, the Bogomolov property of $M$ follows from \cite[Theorem 1.1]{az1}; for general abelian varieties the result was proved in \cite[Theorem 0.1]{bakersilverman}. 
\end{proof}

Note that part (2) of the previous lemma remains true on replacing admissible by
the weaker condition that all composition factors in the respective Galois
groups are non-abelian.  

\begin{proof}[Proof of Proposition \ref{freethings}]
We will construct a tower
 $\ab{K}=F_0\subseteq F_1\subseteq F_2 \subseteq F_3 \subseteq F_4$
 of fields as in Lemma \ref{lem:finitetornonCM} with $e=60d!$,
such that $F\subseteq F_4$. 

Let $F'$ be a Galois extension of $\ab{K}$ with Galois group isomorphic
to $A_n$ or $S_n$. 
If $n\le 5$, then 
$\ex{\gal{F'/\ab{K}}}\mid 60$ and so $F' \subseteq F_1$.
If $n\ge 6$, then $\gal{F'/\ab{K}}$ is either isomorphic to $A_n$,
which is admissible by the second part of Lemma \ref{lem:noAn}, or $\ab{K}$ has a quadratic
extension contained in $F' \cap F_1$. In either case we find that $F'$
is a subfield of $F_2$. We have showed
$\sa{(\ab{K})}\subseteq F_2$. 
If $F'/\sa{(\ab{K})}$ is an extension of degree at most $d$, 
then its normal closure is an extension of $\sa{(\ab{K})}$ of exponent
dividing $d!$. 
Thus $F' \subseteq F_3$ because $d! \mid e$.

In case (1) of the proposition the torsion $\tors{\cG}$ is defined
over $\ab{K}$, which contains all roots of unity. 
The same is true in case (2), since we assume that the elliptic curve
has CM over $K$. 
These two cases of the proposition follow from Lemma
\ref{lem:finitetornonCM}, part (2).
The remaining case (3) follows from part (1) of the same lemma. 
\end{proof}

In order to prove Theorem \ref{thm:intro} it remains to show that the field $F$ does not have the Bogomolov property in any of the three cases described in this theorem. The proof of this is the content of the next two sections.

%%%%%%%%%%%%%%%%%%%%%%%%%%%%%%%%%%%%%%%%%%%%%%%%%%%%%%%%%%%%%%
\section{Small Points on $\GG_m$}

Let $\kappa$ be a field with a fixed algebraic closure $\overline{\kappa}$. If $f\in\kappa[X]$ is a polynomial, then we denote with $\kappa(f)$ the splitting field of $f$ over $\kappa$ inside $\overline{\kappa}$.

In this section we prove the following more precise version of
Theorem \ref{thm:intro} (1). 

\begin{proposition}
\label{prop:Gmexplicit}
Let $K$ be a number field,  $d$  a positive integer,
 and let $F = \left(\left(K^{ab}\right)^{sa}\right)^{(d)}$.
Then $\IG_m(F)/\tors{\IG_m(F)}$ is free abelian, but $F$ does not satisfy the
 Bogomolov property with respect to the Weil height. 
 \end{proposition}

\begin{proof}
We already know that $\GG_m(F)/\tors{\GG_m(F)}$ is free abelian from Proposition \ref{freethings}.

Let $n\ge 2$ be an integer. A result of Osada, cf. \cite[Corollary 3]{Osada}, states that the polynomial
\begin{equation*}
%\label{eq:osada}
f_n = X^n-X-1
\end{equation*}
is irreducible over $\QQ$ and has symmetric Galois group. Moreover, the discriminant of $f_n$ is given by
\begin{equation}\label{disc}
\disc(f)=\pm (n^n + (-1)^n (n-1)^{(n-1)}),
\end{equation}
as explained in \cite{masserdisc}.  

%%%
Let $\Delta$ be the discriminant of $K$, and let $n\ge 5$ 
be an integer with $\Delta\mid n$. Assume that there is a common prime
divisor $p$  of $\Delta$
 and the discriminant of $\IQ(f_n)$. Then $p$ also divides $\disc(f_n)$ and
 $n$. By \eqref{disc} it follows that $p$ divides $n-1$
 as well. This is not possible, and therefore $\Delta$ and
the discriminant of $\IQ(f_n)$ are coprime. 
In particular $K$ and $\IQ(f_n)$ are linearly disjoint over $\IQ$. Hence,
\begin{equation*}%\label{eq:Sn}
 \gal{K(f_n)/K} \cong \gal{\IQ(f_n)/ \IQ}\cong S_n.
\end{equation*}
The group $\gal{K^{ab}(f_n)/K^{ab}} \cong \gal{K(f_n)/(K(f_n)\cap
K^{ab})}$ is a normal subgroup of $S_n \cong \gal{K(f_n)/K}$, since
$K^{ab}$ and $K(f_n)$ are Galois extensions of $K$. It is non trivial and therefore isomorphic to $A_n$ or $S_n$. In both cases we find that 
\begin{equation*}
K(f_n)\subseteq (K^{ab})^{sa}.
\end{equation*}

Let $\alpha$ be a root of $f_n$.  Since $\gal{K^{ab}(f_n)/K^{ab}} \cong A_n$ or $S_n$, we know that $\alpha \not \in K^{ab}.$  In particular, $\alpha$ is not a root of unity, and so $h(\alpha)>0$ by Kronecker's Theorem.  Using basic height properties we have
\begin{equation*}
n\cdot h(\alpha) = h(\alpha^n) = h(\alpha+1) \leq \log2 + h(\alpha) + h(1) = \log2 + h(\alpha),
\end{equation*}
which yields
\begin{equation*}%\label{yey}
 h(\alpha) \leq \frac{\log2}{n-1}.
\end{equation*}

Since we can take $n$ to be arbitrarily large, this gives us that $F$ does not have the Bogomolov property with respect to the Weil height, and our proof is complete.
\end{proof}

%%%%%%%%%%%%%%%%%%%%%%%%%%%%%%%%%%%%%%%%%%%%%%%%%%%%%%%%%%%%%%%
\section{Small Points on Elliptic Curves}%\label{ec}
\subsection{Irreducibility}

For any number field $K$, we denote by $\mathcal{O}_K$ the ring of
integers in $K$. We recall that all number fields lie in a fixed
algebraic closure of $\IQ$.

\begin{proposition}
\label{prop:irred}
Let $K$ be a number field and let $\mfp$ be a maximal ideal in
$\mathcal{O}_K$ with residue characteristic $p\geq 3$. If
$f\in\mathcal{O}_K[T,X]$ is monic in $X$, such that $f\equiv X^n - X^2
+ T^s \imod{\mfp}$, where $n\ge 5$ is odd, $p\mid n$, and $2 \mid s$,
then  the following properties hold true. 
\begin{enumerate}
\item The polynomial $f(T^m,X)$ is irreducible as an element of
  $\IQbar(T)[X]$ for all integers $m\ge 1$. 
\item The group $\gal{F(T)(f)/F(T)}$ is isomorphic to $S_n$ or $A_n$
  for all number fields $F\supseteq K$. 
\end{enumerate}
\end{proposition}
\begin{proof} Let $\IFbar_p$ be an algebraic closure of $\IF_p$. Then
  $X^{n} - X^2 + T^s$ is  irreducible as a polynomial in
  $\IFbar_p(T)[X]$ and separable. For a proof of the former fact
see the first few paragraphs of Section 20 in \cite{Abhyankar}. Our
assumptions on $n$ and $s$ assure that the splitting field of $X^n
-X^2 +T^s \in \IFbar_p(T)[X]$ is a Galois extension of
$\IFbar_p(T)[X]$ with Galois group isomorphic to $A_n$; cf. (II).1 Section 20 in \cite{Abhyankar}.
This polynomial is one of Abhyankar's ``tilde polynomials,'' and the proof of the claimed statement does not require the full classification theorem for finite simple groups. 

The polynomial $f$ from the hypothesis reduces to one of Abhyankar's tilde polynomials, but so do twists by taking powers of $T$. So $f(T^m,X) \imod \mfp$ is irreducible
as an element of $\IFbar_p[T,X]$ for all $m\ge 1$. It follows from the
Gauss Lemma that $f(T^m,X)$ is irreducible as an element of
$\IQbar[T,X]$. This yields part (1).
 
To prove part (2) we set the stage to apply \cite[Lemma
  6.1.1]{FriedJarden05} and use some of the notation introduced before
its statement. Let $F/K$ be a finite extension. Since $\cO_F$ is
integrally closed in $F$, the ring $\cO_F[T]$ is integrally closed in
$F(T)$. Let $\mfP \subseteq \cO_F$ be any prime ideal above
$\mfp$. Then $\mfP [T]$ is a prime ideal of $\cO_F[T]$. The quotient
$\cO_F[T]/\mfP[T]$ is a polynomial ring over a finite field $\IF_q$,
where $q$ is the ideal norm of $\mfP$. Therefore, the quotient field of $\cO_F[T]/\mfP[T]$ is $\widetilde{F}=\IF_q(T)$.

Denote by $L$ the splitting field of $f$ over $F(T)$, and by $G$ the
Galois group of  $L/F(T)$. The integral closure of $\cO_F [T]$ in $L$
contains a prime ideal above $\mfP[T]$. 
We write $\widetilde L$ for the quotient field of the said integral
closure modulo the said prime ideal. 
By construction $L$ contains all roots of $f$. They are integral over $\cO_F [T]$ as $f$ is monic in $X$. Therefore, the reduction $\widetilde f \in \IF_q(T) [X]$ factors completely in $\widetilde{L}[X]$. In particular, $\widetilde{L}$ contains the splitting field $\widetilde{F}(\widetilde{f})$. 

By \cite[Lemma 6.1.1(a)]{FriedJarden05} the extension $\widetilde
L/\widetilde F$ is normal, and a subgroup of $G$ surjects onto its
automorphism group $\aut{\widetilde{L}/\widetilde{F}}$. 
So we have
\begin{equation}
|G| \ge |\aut{\widetilde{L}/\widetilde{F}}|  \ge
|\gal{\widetilde{F}(\widetilde f)/\widetilde{F}}|.
\end{equation}
 On the other hand, $A_n \cong
 \gal{{\IFbar_p(T)}(\widetilde{f})/{\IFbar_p(T)}}$ 
is isomorphic to a subgroup of $\gal{\widetilde{F}(\widetilde{f}) /
   \widetilde{F}}$,
and therefore $|G|\ge n!/2$.
We have proven that $G$ is isomorphic to $S_n$ or $A_n$, concluding our proof.
\end{proof}

\subsection{Hilbert Irreducibility Theorem for Algebraic Groups}

We will use a special case of Dvornicich and Zannier's  Hilbert Irreducibility
Theorem for algebraic groups  to prove the following
proposition.
 
For a subfield $F\subseteq\IQbar$ we let $\cyc{F}$ be  the subfield
of $\IQbar$  obtained by adjoining all roots of unity to  $F$.

\begin{proposition}  
\label{prop:applyzannier}
Let $K$ be a number field. Suppose $f\in K[T,X]$ has degree $n\ge 5$ as a polynomial in $X$ and satisfies the following properties.
\begin{enumerate}
\item  The polynomial $f(T^m,X)$ is irreducible as an element of
  $\IQbar(T)[X]$ for all integers $m\ge 1$. 
\item The group $\gal{F(T)(f)/F(T)}$ is isomorphic to $A_n$ or $S_n$
for all number fields $F\supseteq K$. 
\end{enumerate}
Then for all but finitely many roots of unity $\zeta\in \IQbar$, the specialization $g=f(\zeta,X)$ is irreducible of degree $n$ as an element of $\cyc{K}[X]$, and
$\gal{\cyc{K}(g)/\cyc{K}}$ is isomorphic to $A_n$ or $S_n$. 
\end{proposition}
\begin{proof}
 We notice that the hypotheses assure that
  $\gal{\IQbar(T)(f)/\IQbar(T)}$ is isomorphic to $A_n$ or $S_n$. 
This means that the inclusion $\IQbar(T)\subseteq \IQbar(T)(f)$ of function fields induces a covering
$\pi: Y\rightarrow\IG_m$  of degree $n!/2$ or $n!$,
where $Y$ is a geometrically irreducible curve defined over a finite extension
$\cyc{K}(\alpha)$ of $\cyc{K}$.
By \cite[Proposition 2.1]{Za10} we can factor 
\begin{equation}
\pi = [N]\circ\rho,
\end{equation}
where $\rho:Y\rightarrow\IG_m$ is a rational map that
satisfies Zannier's property (PB) from \cite{Za10}, and $N\geq 1$ is an integer. 

Now $[N]$  comes from a function field $L=\cyc{K}(\alpha)(T)
\supseteq \cyc{K}(\alpha)(T^{N})$. This extension is Galois with group $\IZ/N\IZ$ and also  a quotient of $A_n$ or $S_n$. As $n\ge 5$ we must have $N\leq 2$ and again $\gal{L(f)/L}$ is isomorphic to $A_n$ or $S_n$. 

Let us fix a primitive element $U$ with $L(U) = L(f)$ and an
irreducible polynomial $B\in \cyc{K}(\alpha)[T,X]$ with
$B(T,U)=0$. 
We apply \cite[Theorem 2.1]{Za10} to
$\rho:Y\rightarrow\IG_m$. The specialization
 $B(\zeta,X)$ is thus irreducible in $\cyc{K}(\alpha)[X]$ of degree $[L(f):L]$ for all but finitely many roots of unity $\zeta$. 

Next we use a general specialization principle. More precisely, we
apply \cite[Lemma 1.5]{Volklein96} to specialize the first variable in $B$ to $\zeta$.
 In the reference's notation we take $R=\cyc{K}(\alpha)[T]$ and $A$ to be all roots of $B\in L[X]$,
together with all roots of $f\in K(T)[X]$. 
Let $u\in\IQbar$ denote a root of $B(\zeta,X)$. 
After omitting finitely many $\zeta$, 
the extension $\cyc{K}(\alpha)(u)/\cyc{K}(\alpha)$ is also Galois with Galois group isomorphic to $\gal{L(f)/L}$. 
Moreover,
the specialization $g=f(\zeta,X) \in \cyc{K}(\alpha)[X]$ splits in
$\cyc{K}(\alpha)(u)$. %\todo{Improve.} 
So $\gal{\cyc{K}(\alpha)(g)/\cyc{K}(\alpha)}$ is
a quotient of $\gal{L(f)/L}$ and thus isomorphic to a quotient of
$A_n$ or $S_n$. As $n\ge 5$, the only possibilities are $A_n$, $S_n$, and the trivial group.
However, by \cite[Corollary 1]{DvZa07} and
hypothesis (1) the polynomial $g\in \cyc{K}(\alpha)[X]$ is irreducible
of degree $n$ after omitting finitely many $\zeta$. We may thus rule out the
trivial group. The Galois group of $\cyc{K}(g)/\cyc{K}$ is
isomorphic to a subgroup of $S_n$ and contains a subgroup isomorphic
to $A_n$ or $S_n$, and therefore $\gal{\cyc{K}(g)/\cyc{K}}$ must also be of this type.
\end{proof}

\subsection{Proof of Theorem \ref{thm:intro} for elliptic curves}

Here we will prove the following result which implies parts (2) and
(3) of Theorem \ref{thm:intro}
\begin{proposition}
\label{prop:ellcurveexplicit}
Let $E$ be an elliptic curve defined over a number field $K$.
We suppose 
\begin{enumerate}
\item that $E$ has CM over $K$
\item or that $E$ does not have CM and $K=\IQ$.
\end{enumerate}
Let
$d\ge 2$ be  a positive integer,
 and let $F = \left(\left(K^{ab}\right)^{sa}\right)^{(d)}$.
Then $E(F)/\tors{E(F)}$ is a free abelian group that does not satisfy the
 Bogomolov property with respect to the N\'eron-Tate height. 
\end{proposition}

\begin{proof}
Let $E,K,$ and $F$ be as in the theorem. 
We will work implicitly with  a fixed short
Weierstrass equation  for $E$ with coefficients in $\cO_{K}$.

Let us suppose for the moment that $\mfp$ is a prime ideal in $\cO_K$ of norm $q=p^k$, with
$p\geq 5$, where $E$ has good, supersingular reduction $\widetilde E$.  By \cite[Chapter 13, Theorem 6.3]{Husemoeller} a power $\frob_q^{\nu}$ of
the  Frobenius endomorphism  $\frob_q$ of $\widetilde E$ equals
$[p^{\mu}]$, the multiplication-by-$p^{\mu}$ endomorphism, where $\nu,\mu \in\IN$. Taking the degree yields $q^{\nu}=p^{2\mu}$,
and so
\begin{equation*}
\frob_q^{2\nu}=[p^{2\mu}]=[q^{\nu}].
\end{equation*}
The first coordinate in the multiplication-by-$q^{\nu}$ morphism of $E$ is represented by a  quotient $a/b$ of  polynomials
 $a=X^{q^{2\nu}}+\cdots,b=q^{2\nu}X^{q^{2\nu}-1}+\cdots\in\cO_K[X]$. 
By the previous paragraph we have $\frac{a}{b} \equiv X^{q^{2\nu}} \imod \mfp$.
Hence
\begin{equation*}
a\equiv X^{q^{2\nu}} \imod \mfp \quad \text{and} \quad b\equiv 1 \imod \mfp,
\end{equation*}
as $\IF_q[X]$ is factorial. We define the auxiliary
polynomial
\begin{equation*}
  f = X^{2p} a - (X^2-T^2)b \in \cO_K[T,X].
\end{equation*}
It is monic in $X$ of degree $q^{2\nu}+2p$ as $\deg(b)< q^{2\nu}$ and 
\begin{equation*}
  f\equiv X^{q^{2\nu}+2p} - (X^2-T^2) \imod \mfp. 
\end{equation*}
Thus $f$ satisfies the hypothesis of Proposition \ref{prop:irred}. By
Proposition \ref{prop:applyzannier} 
$g(X) = f(\zeta,X)$ is irreducible in $\cyc{K}[X]$
for all but finitely many roots of unity $\zeta$,  and moreover
 $\gal{\cyc{K}(g)/\cyc{K}}$ is isomorphic to $S_{n}$ or $A_n$ with $n=q^{2\nu}+2p$.

Let $\alpha$ be a root of $g$  and let $\beta\in\IQbar$ such that
$Q=(\alpha,\beta)$ lies on  $E$. 
By properties of $g$ we have $\alpha\in (\cyc{K})^{sa}$. Now
$\cyc{K}\subseteq K^{ab}$, and the Galois group of $\cyc{K}(g) / K^{ab}\cap \cyc{K}(g)$ 
is normal in $A_n$ or $S_n$ with abelian quotient, so it is again
$A_n$ or $S_n$ as $n=q^{2\nu}+2p\ge 5$. 
We conclude $\alpha\in K^{ab}(g)\subseteq (K^{ab})^{sa}$.	
Solving for $\beta$  merely involves taking a square root, so 
 $\beta\in ((K^{ab})^{sa})^{(2)} \subseteq F$ with $F$ as in the
hypothesis. Hence
$Q\in E(F)$. 

Since $\alpha$ has degree
$q^{2\nu}+2p$ over $\cyc{K}$, we have $b(\alpha)\neq 0$ and $\alpha\not=0$. 
Set $[q^{\nu}]Q=(\alpha',\beta')$. Then the choice of $f$ yields
\begin{equation*}
%\label{eq:xp}
  \alpha^{2p}\alpha' = \alpha^2 - \zeta^2.
\end{equation*}
Fundamental properties of the N\'{e}ron-Tate height $\widehat{h}$ on $E$
and of the absolute logarithmic Weil height $h$ imply
\begin{align*}
%\label{eq:nthbound}
  q^{2\nu}\widehat{h}(Q)  &= \widehat{h}((\alpha',\beta')) \leq h(\alpha') + c \nonumber \\ 
	 &= h(\alpha^{2-2p}- \zeta^2\alpha^{-2p}) + c \leq h(\alpha^{2p-2})+h(\alpha^{2p})+c+\log{2}  \\
	 &=(4p-2)h(\alpha)+c+\log{2}\leq (4p-2)(\widehat{h}(Q)+c)+c+\log{2} \nonumber \\
	 &\leq 4p\widehat{h}(Q)+4cp+\log{2}
\end{align*}
where $c$ depends only on $E$ and compares the
N\'eron-Tate height to the  height of the $x$-coordinate. As
$q^{2\nu}>2p$ this gives the inequality
\begin{equation*}
 \widehat{h}(Q)\leq\frac{4cp+\log{2}}{q^{2\nu}-4p}
\end{equation*}
Observe that the right-hand side tends to $0$ as $p\rightarrow\infty$. 

In case (2), where the elliptic curve $E$ does not have CM, but is defined over $\QQ$, we know that $E$ has
supersingular reduction at infinitely many primes by Elkies's Theorem
\cite{Elkies87}. 
In case (1),  infinitude follows 
from more classical considerations; cf. \cite[Chapter 13, Theorem 12]{Lang:elliptic}.  Hence the construction above yields a sequence of points
 $Q_1,Q_2,\ldots\in E(F)$ 
with N\'eron-Tate height tending to $0$. This sequence contains
infinitely many pair-wise distinct members as the lower bound in
$[\cyc{K}(Q):\cyc{K}]\ge q^{2\nu}+2p$ tends to $+\infty$ in $p$.  To prove that $F$ does not have the Bogomolov property
relative to $\widehat{h}_E$ it remains to
 show that there are infinitely many non-torsion points among the constructed points $Q_1,Q_2,\dots$.

This is quite easy in case (1), where $E$ has CM. In fact, none of
the 
$Q=(\alpha,\beta)$ constructed above have finite order.
Indeed, let us assume the contrary. The field $K$, and a
fortiori $\cyc{K}$, contains the
CM field of $E$. So $\cyc{K}(\alpha)$ is an abelian extension of $\cyc{K}$.
On the other hand $\cyc{K}(\alpha)\subseteq \cyc{K}(g)$, and the latter
is a Galois extension of $\cyc{K}$ with Galois group isomorphic to
$A_{n}$ or $S_{n}$. 
We use  $n\ge 5$ again to find $[\cyc{K}(\alpha):\cyc{K}]\le 2$ as any abelian
quotient of $A_n$ or $S_n$ has order at most $2$. This contradicts 
$[\cyc{K}(\alpha):\cyc{K}]=q^{2\nu}+2p > 2$, so $Q$ has infinite order
and thus $\widehat h(Q)>0$ by Kronecker's Theorem.  

Finally, if $E$ does not have CM, we already know that $E(F)$ contains only finitely many torsion points, by Proposition \ref{freethings}, completing our proof.
\end{proof} 

\subsection{Open problems}
As usual in this paper let $\mcG$ be $\mathbb{G}_m$ or an elliptic curve defined over a number field $K$. In proving  Propositions \ref{prop:ellcurveexplicit} and \ref{prop:Gmexplicit}, we have seen that
Criteria A and B (Propositions \ref{fintors} and \ref{finreltors}) are not valid if we replace the phrase \textit{is free abelian}
with \textit{satisfies the Bogomolov property with respect to $\mcG$}. 
Inspired by  results of May \cite{May72}, 
we state the following open questions, where the second one is a generalization of \cite[Problem 1.4]{ADZ:propB}.

\begin{question}
Let $F$ be an algebraic extension of $K$ such that every finite extension of $F$ has the Bogomolov property with respect to $\mcG$.
\begin{enumerate}
\item Does $F^{(d)}$ satisfy the Bogomolov property with respect to $\mcG$ for every integer $d$?
\item Does $F^{ab}$ satisfy the Bogomolov property with respect to $\mcG$ if in addition every finite extension of $F$ contains only finitely many roots of unity?
\end{enumerate}
\end{question}

The Bogomolov property is not preserved under finite field
extension. But perhaps it holds under additional restrictions.

\begin{question}
Let $F_0/K$ be an algebraic extension such that $F_0$ has the
Bogomolov property with respect to $\mcG$. Does a finite extension
field $F/F_0$, with $F$ Galois over $K$ and with
$\tors{\mcG(F)}\setminus\tors{\mcG(F_0)}$ a finite set necessarily satisfy the Bogomolov property with respect to $\mcG$?  
\end{question}

Note that none of the assumptions can be removed. The necessity of the finiteness of $\tors{\mcG(F)}\setminus\tors{\mcG(F_0)}$ is due to the failure of the Bogomolov property in the extension $\mathbb{Q}^{tr}(i)/\mathbb{Q}^{tr}$. The necessity of a Galois extension $F/K$ comes from a variation of this extension; cf. \cite[Example 1]{May72}.

%%%%%%%%%%%%%%%%%%%%%%%%%%%%%%%%%%%%%%%%%%%%%%%%%%%%%%%%%%%%%%%
\bibliography{ghp}{}
\end{document}